\documentclass[11pt,final]{article}

\usepackage[english]{babel}
\usepackage[T1]{fontenc}
\usepackage{mathpazo,microtype}

\usepackage{amsmath,amsthm,amssymb,amsfonts,bm,mathtools, enumitem, cancel} 

\usepackage[final]{hyperref} 

\usepackage[final]{graphicx}
\usepackage{import,xcolor} 

\usepackage[doi=false,isbn=false,url=false,style=alphabetic]{biblatex}
\usepackage{ifdraft} 
\usepackage[obeyDraft,textwidth=35mm]{todonotes}

\usepackage{tikz}
\usetikzlibrary{calc,trees,positioning,arrows,chains,shapes.geometric,%
    decorations.pathreplacing,decorations.pathmorphing,shapes,%
    matrix,shapes.symbols}
\tikzstyle{rect} = [rectangle, rounded corners, 
minimum width=4cm, 
minimum height=1cm,
text centered, 
draw=black, 
fill=white!30]
\tikzstyle{arrow} = [thick,->,>=stealth]

\newtheorem{Thm}{Theorem}
\newtheorem{Cor}[Thm]{Corollary}

\newtheorem{thm}{Theorem}[section]

\newtheorem{lem}[thm]{Lemma}

\theoremstyle{remark}
\newtheorem{rmk}[thm]{Remark}
\theoremstyle{definition}
\newtheorem{Def}[thm]{Definition}

\newtheorem{manualtheoreminner}{Theorem}
\newenvironment{manualtheorem}[1]{%
  \renewcommand\themanualtheoreminner{#1}%
  \manualtheoreminner
}{\endmanualtheoreminner}

 \newcommand{\hol}{ \operatorname{Hol}}

\title{Cocycle Stability of Linear Automorphisms and their Jointly Integrable
Perturbations}

\author{Ignacio Correa}

\date{}

\begin{document}

\maketitle

\begin{abstract}
  We show cocycle stability for linear maps with a weak irreducibility
  condition and their jointly
  integrable perturbations.
\end{abstract}

\section{Introduction} 
For a dynamical system $f$ we are interested in
solving for which $\varphi:\mathbb{T}^d \rightarrow \mathbb{T}^{d}$ the cohomological
equation
\begin{equation}
  \label{eq:cohomological}
  \varphi(x)=u(fx)-u(x)+ \mathrm{const.} \tag{$\star$}
\end{equation}
has a solution $u$ (in other words, we
want to see when $\varphi$ is cohomologus to a constant).

The classic result for this question is Livsic, which for hyperbolic
systems states that having trivial periodic data is a sufficient and
necessary condition. For linear ergodic automorphisms Veech showed
that this is also the case \cite{Veech1986PeriodicPointsInvariant}.

One can also show, using Fourier series, that for Diophantine
translation of the torus there is no restriction at all (other than
perhaps some extra regularity) as for which $\varphi$ accept solutions.

For partially hyperbolic systems (which a priori might not have
periodic points) Katok and Kononenko came up with a natural necessary condition: to have trivial periodic cycle
functionals, that is \(\varphi\) should have trivial holonomies along cycles made out
of stable and unstable segments (see \S~\ref{holonomy} or
\cite{Katok.Kononenko1996CocycleStabilityPartially} for details).

We will show that, in some cases, this is restriction is sufficient, which immediately
implies \emph{cocycle stability} in the sense of \cite[Definition
1]{Katok.Kononenko1996CocycleStabilityPartially}. That is, the set of
coboundaries is closed, which would be the case since then we would have
\begin{equation*}
  \biggl\{ \varphi : \varphi = u\circ f - u \biggr\}  = \biggl\{ \int\varphi=0 \biggr\} \cup \bigcup_{\gamma \text{ $su$-cycle}}  \biggl\{
    PCF_{\gamma}\varphi=0 \biggr\}.
\end{equation*} More specifically we will show that the $C^R$
functions are $C^r$-sable.

Results in this direction are known when $f$ has some accessibility
condition \cite{Wilkinson2013CohomologicalEquationPartially,
  Katok.Kononenko1996CocycleStabilityPartially}. So instead in  these notes we
will work with the spiritual opposite of this case: the \emph{jointly
  integrable} case (by this we mean that the $su$-distribution
$E^s\oplus E^u$ is integrable, see definition~\ref{jointlyint}).

While some other minor results can be established with the ``naive''
idea of our technique
(see \S~\ref{toy}), our most interesting application would be to linear
maps and their (jointly integrable) perturbations.

\begin{Thm}\label{main}Given a chosen regularity $r$ (and a central
  dimension $c$) there is a minimum regularity $R$ so that for any
  partially hyperbolic $F\in SL(d, \mathbb{Z})$ \emph{Katznelson irreducible} and strongly $R$-bunched  automorphism of the torus, if a jointly integrable
  diffeomorphism $f: \mathbb{T}^d\to \mathbb{T}^d$ is $C^{R}$-close enough to
  $F$ then it is cocycle stable with the following
  regularities:

  If $\varphi$ is $C^{R}$ with
  $\operatorname{PCF}_{\gamma}\varphi=0$ for any $su$-cycle $\gamma$, then the
  cohomological equation
\begin{equation*}
  \varphi(x)=u(fx)-u(x)+ \mathrm{const.} \tag{\ref{eq:cohomological}}
\end{equation*}
has a $C^r$ solution $u$.
\end{Thm}

\emph{Katznelson irreducible} is the (very weak) irreducibility
condition needed to apply Katznelson's lemma \cite[Lemma
3]{Katznelson1971ErgodicAutomorphismsTn} and is explicitly defined
in definition~\ref{diophantine condition}. There we
show that this condition is actually necessary for solving \eqref{eq:cohomological}.

For the purposes of potential applications, in \S~\ref{optimal}, we
also give explicit values for $R$.

Our main tool would a $\mathbb{Z}^d$-action on a central leaf called
\emph{central translations}, introduced by F. Rodriguez-Hertz in
\cite{RodriguezHertz2005StableErgodicityCertain} and defined here in
\S~\ref{sec:trans}. Following that paper the objective would be to show that these enjoy a
Diophantine property (\S~\ref{linear-thm}) in the linear case and then, using Moser's
techniques \cite{Moser1990CommutingCircleMappings}, we show that they
can be linearized in the non-linear case (\S~\ref{fede linearization}).

In here we are able to show this Diophantine property in more
generality, and since
\cite[\S~6.3]{RodriguezHertz2005StableErgodicityCertain} doesn't use
any of their other hypotheses we have the following result:
\begin{Cor}\label{conjugacy}

  If a partially hyperbolic automorphism $F\in SL(d, \mathbb{Z})$ is Katznelson irreducible and $(9c+4)$-bunched,
  then any $C^{9c+4}$ jointly integrable diffeomorphism $f$ which is $C^{6c+1}$-close
  to $F$ is ergodic and topologically conjugated to $F$.

  In here $c$ is the dimension of the center distribution for the
  partial hyperbolic splitting.

\end{Cor}

In
\cite[Theorem 5.1]{RodriguezHertz2005StableErgodicityCertain} it is shown that, in
the pseudo-Anosov case (which is stronger than Katznelson
irreducibility) with center dimension $2$, there is a
dichotomy on the accessibility classes:
either $E^s\oplus E^u$ is integrable or $f$ is accessible\footnote{Their
  hypothesis of having eigenvalues of modulus one is not necessary for
this result.}. In the general
accessible case Wilkinson
\cite{Wilkinson2013CohomologicalEquationPartially} showed the cocycle
stability. So we get that these automorphisms are
``stably cocycle stable'':
\begin{Cor}
  If $F\in SL(d,\mathbb{Z})$ is pseudo-Anosov, partially hyperbolic with
  splitting $S\oplus C \oplus U$ with $\dim C =2$, and strongly $R$-bunched, then
  any $f$ which is $C^{R}$-close to $F$ (with $R$ big enough) is cocycle stable.
\end{Cor}
The $R$ in this theorem as well as the regularity of the solution of
\eqref{eq:cohomological} are as established in \S~\ref{optimal}.

The idea of the proof is that we can solve \eqref{eq:cohomological}
along a central leaf
using the Diophantine property of the central translations
(\S~\ref{sec:solved-central}). Then using the holonomy functionals
defined in \S~\ref{holonomy} we can extend this solution to the whole
torus (\S~\ref{extend}).

In \S~\ref{toy}, which can be read independently of the rest of this
notes, we illustrate this idea of using Diophantine dynamics to solve
\eqref{eq:cohomological} along the center and then extend using
holonomy functionals in a simpler case.

\section{Preliminaries}

\subsection{Partial Hyperbolicity}

\begin{Def}

We say that a diffeomorphism $f:M \to M$ is \emph{partially hyperbolic} if
there is an invariant splitting $TM = E^s\oplus E^c \oplus E^u$ so that (for
some metric) 
\begin{align*}
  \lVert D^s_xf \rVert &< 1,  &m(D^u_xf) &> 1, \\
  \lVert D^s_xf \rVert &< m( D^c_xf ), &m(D^u_xf)  &> \lVert D^c_xf \rVert,
\end{align*} where $m(A)$ is the conorm of $A$ and $D^{\sigma}f = Df
|_{{E^{\sigma}}}$.
\end{Def}

For general partially hyperbolic systems $E^s$ and $E^u$ are
integrable into invariant foliations $\mathcal{W}^{ss}$ and $\mathcal{W}^{uu}$. In our case, where $f$ is
a perturbation of a linear map, also $E^c$, $E^s\oplus E^c$, and $E^c\oplus E^u$
are integrable into invariant foliations
 $\mathcal{W}^c$, $\mathcal{W}^{cs}$, and $\mathcal{W}^{cu}$ \cite{Hirsch.Pugh.ea1970InvariantManifolds}.

\begin{Def}\label{jointlyint}

  We say that a partially hyperbolic diffeomorphism $f$ is\emph{ jointly
  integrable} if there is an invariant foliation $\mathcal{W}^{su}$ tangent to
  $E^s\oplus E^u$. We call the leaves of $\mathcal{W}^{su}$ of \emph{$su$-leaves}.

\end{Def}

\begin{rmk}\label{suregular}
This foliation $\mathcal{W}^{su}$, when it exist, has $C^r$ leaves. Indeed,
locally we can express the leaves as graphs and we can use Journé's
lemma \cite{Journe1988RegularityLemmaFunctions} on the functions giving
these graphs to prove this regularity. Furthermore, $\mathcal{W}^{su}(x)$ only
needs to be a set saturated by stable and unstable leaves and
transversal to the center leaves. But we won't use this fact in any
relevant way, the only important thing is that the projection along
stable leaves and along unstable leaves commute.
\end{rmk}

Of course linear maps $F\in SL(d,\mathbb{Z})$ (thought as maps of the torus
$\mathbb{T}^d$) are jointly integrable as the
invariant splitting should be a splitting $\mathbb{R}^d=S\oplus C \oplus U$ into vector
subspaces.

In our theorems we ask for hypothesis on $F$ with respect to an implicitly given partial
hyperbolic splitting.
Whenever we ask for $f$ to be jointly integrable, we mean with respect
to the splitting that comes from perturbing the given splitting.

Asking for stronger conditions (called bunching conditions) in the
domination of center it is known
\cite{Wilkinson2013CohomologicalEquationPartially} that relevant
holonomy maps have certain amount of regularity.

\begin{Def}

  We say that $f$ partially hyperbolic is \emph{$R$-bunched} if (for
  some metric) for
  every $x$
\begin{align*}
  \lVert D^s_xf \rVert& < m( D^c_xf )^R,  &m(D^u_xf) &> \lVert D^c_xf \rVert^R, \\
  \lVert D^s_xf \rVert &< m( D^c_xf ) \lVert D^c_xf \rVert^{-R},  &m(D^u_xf) &> \lVert D^c_xf \rVert m( D^c_xf )^{-R}.
\end{align*}

If further 
\begin{equation*}
  \max \{ \lVert D^s_xf \rVert, m(D^u_xf)^{-1} \} < \min \{  m(D^c_xf)^R, \lVert D^c_xf \rVert^{-R} \}
\end{equation*} we say that $f$ is \emph{strongly $R$-bunched}.
\end{Def}

Note that the bounds for partial hyperbolicity and for bunching are
all open in the $C^1$ topology, so if $f$ is close enough to a
 linear map $F$ satisfying this conditions so will $f$.

On the other hand, for a linear map $F$ is easy to show that the partial hyperbolicity
condition and the bunching conditions are equivalent to analogous
conditions when replacing norms and conorms with the biggest and
smallest (absolute value of) eigenvalue of the respective subspace. In
this case if $F|_C$ only has eigenvalues of modulus one then it is
$R$-bunched and strongly $R$-bunched for any $R$.

In this notes partial hyperbolicity and jointly integrability are
always assumed.

 \subsection{Central Translations}\label{sec:trans}
Solving \eqref{eq:cohomological}
without any hypothesis on \(\varphi\) usually requires a certain Diophantine
property. In our case this property is hidden behind an auxiliary
construction called the \emph{central translations}.

 Is easy to see that in our case (jointly integrable perturbations of
 linear maps) every pair of $su$-leaf and central leaf intersect
 exactly once in the universal cover. Fix a ``favorite'' central leaf
 $\mathcal{C} = \mathcal{W}^c(0)$ of a fix point $0$.

 We define an action  $T_n: \mathcal{C} \to \mathcal{C}$ of $\mathbb{Z}^d$ in $\mathcal{C}$ by 
\begin{equation*}
  \{ T_nx\} = \mathcal{C} \bigcap \mathcal{W}^{su}(x+n)
\end{equation*}
where, by abuse of notation, we are thinking of $\mathcal{C}, \mathcal{W}^{su}(x+n) \subset \mathbb{R}^d$
as living in the universal cover.

\begin{figure}[h]
    \centering
    \def\svgwidth{0.5\columnwidth}
    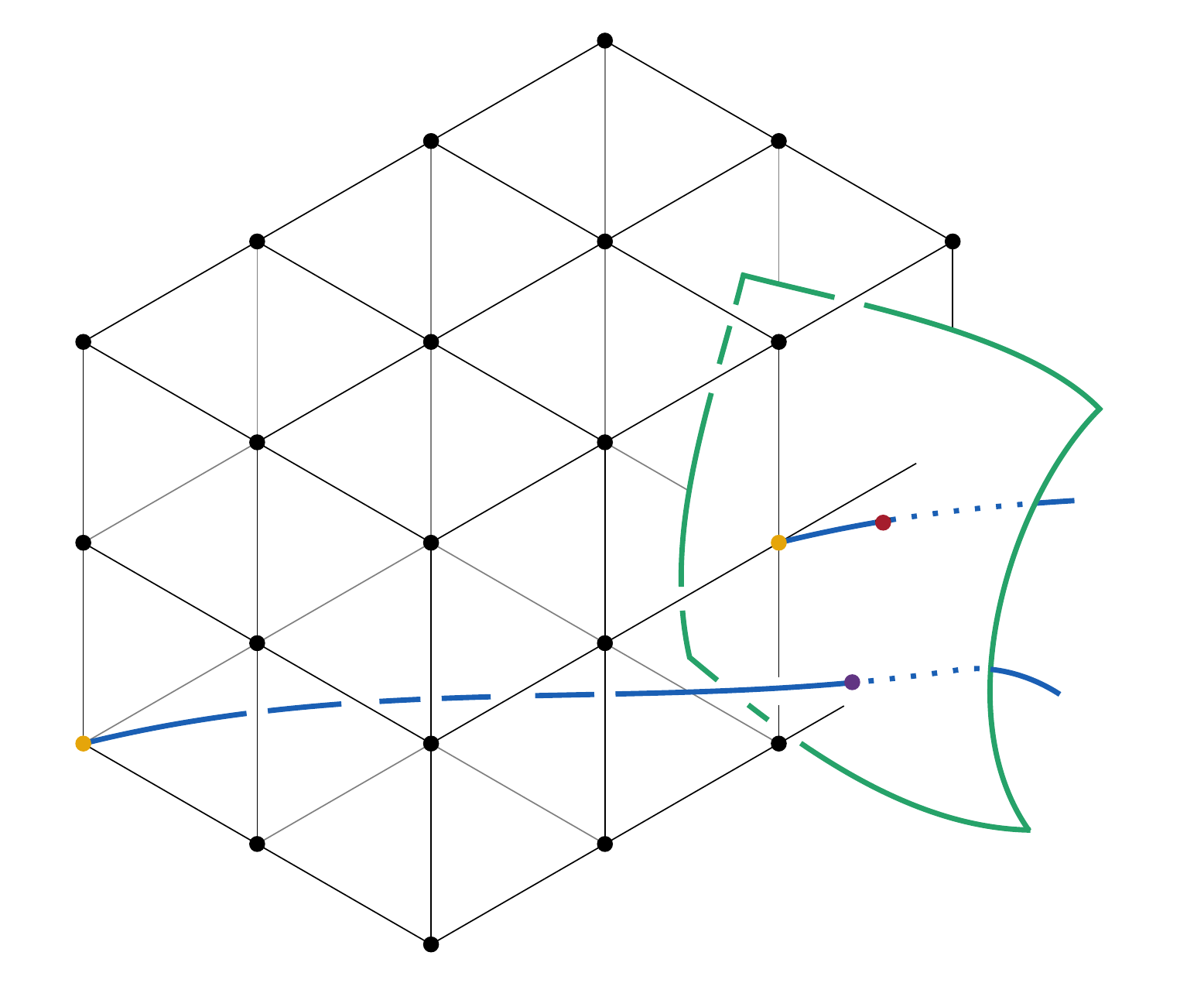
\end{figure}

We refer to these as \emph{central translations}. In \S~\ref{linear-thm}
we will show that these are, in a certain sense, Diophantine and so
the perturbations can be linearized (\S~\ref{fede linearization}) and we
will be able to use this Diophantine property to solve the cohomological equation along the center in
\S~\ref{sec:solved-central}.

Since the intention is to linearize these translations we start here by
projecting them to a vector space 
\begin{equation*}
  \tilde{T}_n= \Pi \circ T_n \circ \Pi^{-1} : C \to C
\end{equation*} where $C \subset \mathbb{R}^d$ is the central subspace for the nearby
linear map $F$ and $\Pi: \mathcal{C} \to C$ is the projection to this subspace along
the $su$-leaves of $f$. See
\cite[Proposition~B.1]{RodriguezHertz2005StableErgodicityCertain} to
see that $\Pi$ is indeed well defined and invertible.

This $\Pi$ shouldn't be confused with $\pi^c:\mathbb{R}^d\to \mathcal{C}$  the projection to
our favorite central leaf for $f$
\begin{equation*}
  \{ \pi^cx\} = \mathcal{C} \bigcap \mathcal{W}^{su}(x).
\end{equation*}

Notice that $\pi^c$ and $\Pi$ are both ``reasonable'' holonomy maps, and
that $T_n(x) = \pi^c(x+n)$. So according to
\cite{Pugh.Shub.ea1997HolderFoliations} we have 
\begin{lem}

  If $f$ is $R$-bunched and $C^R$ then $T_n$ and $\tilde{T}_n$ are uniformly
  $C^R$, and $\pi^c$ is uniformly $C^R$ along the leaves of $\mathcal{W}^c$.

\end{lem}

We describe what uniformly $C^R$ means right after lemma~\ref{holreg}.

\subsection{Holonomy Functionals}\label{holonomy}
 
Let us define the \emph{holonomy functionals}.  Consider $x$ a point and $x^+\in \mathcal W^{ss}(x)$, and 
$x^-\in \mathcal W^{uu}(x)$. In each case we define the functional
$\operatorname{Hol}$ in the following ways
\begin{gather*}\operatorname{Hol}_x^{x^+} \varphi = \sum_{n=0}^\infty \varphi (f^nx) - \varphi
(f^nx^+),\\
\operatorname{Hol}_x^{x^-} \varphi =- \sum_{n=1}^\infty \varphi (f^{-n}x) - \varphi (f^{-n}x^-) \end{gather*}
and if $\gamma = (x_0,x_1,\dots,x_k)$ is a \emph{$su$-sequence}, by that we mean
$x_{j}\in \mathcal W^{ss}(x_{j-1})$ or
$x_{j}\in \mathcal W^{uu}(x_{j-1})$, we extend the functional
$$\hol_\gamma \varphi = \sum_{j=1}^k \hol_{x_{j-1}}^{x_j}
\varphi.$$ In the case where $\gamma$ is \emph{periodic}, that is $x_0=x_k$,
this is the \emph{periodic cycle functional}
$PCF_\gamma \varphi = \operatorname{Hol}_\gamma \varphi$.

The condition $PCF_\gamma \varphi = 0$ for all $su$-sequence $\gamma$ is a natural
necessary condition for solving the cohomological equation
\eqref{eq:cohomological}. In such case we say that $\varphi$ has \emph{trivial
  periodic cycle functionals}, and in this notes we will always assume that
$\varphi$ satisfies this condition.
 
So, under this hypothesis, if $y\in \mathcal{W}^{su}(x)$ (recall that in
this notes we consider $f$ being jointly integrable) we write
$$\hol_x^y \varphi
=\hol_\gamma \varphi$$ where $\gamma = (x_0=x,x_1,\dots,x_k=y)$ is a
$su$-sequence going from $x$ to $y$ (the trivial periodic cycle
hypothesis makes this well defined, regardless of the choice of $\gamma$).

Formally $-\sum\varphi(f^nx)$ and $\sum \varphi(f^{-n}x)$ solve
\eqref{eq:cohomological}, so, naively, we can think of
``$\hol_x^y\varphi=u(y)-u(x)$''. This, in a sense, shows that this functional
solves the equation along $\mathcal{W}^{su}$.

Below we establish some trivial properties of these functionals to
reference in the future (one of which formalizes the idea of the
previous paragraph).

\begin{lem}

  If $y,z\in \mathcal{W}^{su}(x)$ then $\hol$ satisfies the following properties
  \begin{description}
  \item[Fundamental theorem of dynamics]  
\begin{equation*}\label{fundamental}
  \hol_x^y (u\circ f - u)= u(y)- u (x) \tag{FT}
\end{equation*} or, equivalently, 
\begin{equation*}
   \left( \hol_{fx}^{fy}-\hol_x^y\right) \varphi = \varphi(y)- \varphi(x) \tag{\ref{fundamental}}
 \end{equation*}
 \item[Additivity]
  \begin{equation}\label{additivity}
  \hol_x^z\varphi= \hol_x^y\varphi+\hol_y^z\varphi. \tag{Add.}
\end{equation}
  \end{description}

\end{lem}

We establish too the regularity that we will need for these functionals.
\begin{lem}\label{holreg}
 If $f$ and $\varphi$ are $C^R$ then $$\hol
 ^x_{x_0}\varphi: \mathcal{W}^{su}(x_0) \to \mathbb{R}$$ is uniformly $C^R$.
 
  If further $f$ is strongly $R$-bunched then $$\hol
  ^{\pi^cx}_{x}\varphi: \mathcal{W}^c(x_0) \to \mathbb{R}$$ is uniformly $C^R$.

\end{lem}

By uniformly $C^R$ we mean that all partial derivatives (and the
Hölder constants for the highest order derivatives if $R$ is not an
integer) are bounded both along $\mathcal{W}^{\sigma}(x_0)$ and as we vary $x_0\in \mathbb{T}^d$.

\begin{proof}

For first part, if we restrict our attention to $\mathcal{W}^{ss}(x)$ and
$\mathcal{W}^{uu}(x)$ then it is just a simple calculation. Then using Journé
\cite{Journe1988RegularityLemmaFunctions} we have the result.

For the second part strongly $R$-bunched implies that the $C^R$ cocycle dynamics $\tilde{f}:M\times
  \mathbb{R} \to M\times
  \mathbb{R} $ 
\begin{equation*}
  \tilde{f}(x,t)=(fx,t+\varphi(x))
\end{equation*}
 is partially hyperbolic $R$-bunched (see \cite{Wilkinson2013CohomologicalEquationPartially}).

If $y\in \mathcal{W}^{ss}(x)$ we can show that
 $d(\tilde{f}^n(x,0),\tilde{f}^n(x,\hol_{x}^y\varphi))$ goes to $0$
 exponentially fast with and easy computation. We have an analogous
 result for
  $y\in \mathcal{W}^{uu}(x)$, so $$\mathcal{W}^{su}_{\tilde{f}}(x,t)= \left\{ (y,
     \hol_{x}^y\varphi+t) :
     x\in \mathcal{W}^{su}_f(x) \right\}.$$

   So, since the $C^R$ holonomy map from $\mathcal{W}^c_f(x_0)\times \mathbb{R}$
   to 
   $\mathcal{W}^c_f(\pi^cx_0)\times \mathbb{R}$ is 
\begin{equation*}
  (x, t) \mapsto \left( \pi^cx, \hol_{x}^{\pi^cx}\varphi+t\right),
\end{equation*} we get that $\hol_{x}^{\pi^cx}\varphi$ is $C^R$ as a map from
$\mathcal{W}^{c}_f(x_0)$ to $\mathbb{R}$.
\end{proof}

\section{Central Translations are Diophantine}

For most of this section
we will restrict our attention to the case where $f=F\in
SL(d, \mathbb{Z})$ is a linear automorphism. To emphasize the difference within the
linear and not linear case (as both cases will eventually coexist) we
will call the splitting for the linear map $F$ of $S\oplus C \oplus U$ and use
$a$ to refer to the points of $\mathbb{T}^d$ or $\mathbb{R}^d$ when $F$ is the map
acting there.

\subsection{Katznelson Irreducibility}\label{katznelson}

\begin{Def}
  
  We say that $F\in SL(d, \mathbb{Z})$ partially hyperbolic with
  splitting $S\oplus C \oplus U$ is \emph{Katznelson irreducible} if $C \cap \mathbb{Z}^d
  = \{0\}$.

\end{Def}

We remark that this definition depends on the choice of the splitting,
and in many situations the optimal splitting would be when $C$ is the
generalized eigenspace of eigenvalues of modulus $1$. In the latter case
Katznelson irreducibility is just ergodicity (see below).

 Notice that by partial hyperbolicity
  $F|_C$ and $F_{S\oplus U}$ have no common eigenvalues, so this condition
  is enough to apply Katznelson's lemma
  \cite[Lemma 3]{Katznelson1971ErgodicAutomorphismsTn}.

To compare with more algebraic definition of irreducibility we give
the following equivalence:

\begin{lem}\label{poly}

  A partial hyperbolic automorphism $F\in SL(d, \mathbb{Z})$ is Katznelson
  irreducible if and only there is no rational factor
  $\eta(x) \in  \mathbb{Q} [x]$ of its characteristic polynomial having only
  eigenvalues of $F|_C$ as roots.

\end{lem}

\begin{proof}

  Indeed if $\bm{p} \in C \cap \mathbb{Z}^d \setminus\{0\}$ then
  $Q =\operatorname{span}\{F^i\bm{p}\}_{i \in \mathbb{Z}} \subset C$ is a rational
  subspace and the characteristic polynomial of $F|_Q$ has only
  eigenvalues of $F|_C$ as roots.

  On the other direction $\ker \eta (F) \subset C$ is a rational subspace and
  so it has integer vectors.
\end{proof}

From this lemma the following implications are obvious for $F\in
SL(d,\mathbb{R})$ 
\begin{equation*}
  \text{Irreducibility} \implies \text{Katznelson Irreducibility}
  \implies \text{Ergodicity}.
\end{equation*} and none of the reverse implications are true, except
in the case where $F|_C$ has only eigenvalues of modulus $1$, in which
case
\begin{equation*}
  \text{Katznelson Irreducibility} \iff \text{Ergodicity}
\end{equation*} since a rational polynomial can have only roots of
modulus $1$ if those roots are roots of unity.

We show now that this condition is actually necessary to solve the
cohomological equation under our hypotheses.
\begin{thm}\label{reciprocal}

  If for $F\in SL(d,\mathbb{Z})$ the equation \eqref{eq:cohomological} can be
  solved for any \(\varphi\) with trivial periodic cycle functionals, then it is Katznelson irreducible.

\end{thm}

\begin{proof}

  If $F$ is not Katznelson irreducible, then take $\eta(x)\in \mathbb{Q}[x]$ the
  factor of the characteristic polynomial $\chi_{F}(x)$ with only eigenvalues of
  $F|_{C}$ as roots.

  Then $V=\ker \eta(F) \subset C $ and
  $W=\ker \frac{\chi_{F}}{\eta}(F) \supset S\oplus U$ are invariant sub-tori.  By taking the quotient
  $\tilde{V}= \mathbb{T}^d/ W$ we get a torus finitely covered by $V$.

  By taking a power of $F$ we can assume that there are two different
  fixed points in $V$ that project to different points in $\tilde{V}$.
  By defining any function in $\tilde{V}$ that gives different values
  to such projections, we can extend that function to $\mathbb{T}^d$ by making
  it constant along $W$.

  Such function clearly has trivial periodic cycle functional while it cannot be
  cohomologus to a constant since it gives different values to two
  different fixed points.
\end{proof}

\subsection{Using Katznelson to Show a Diophantine Property}\label{linear-thm}

 For $v\in \mathbb{R}^d= S\oplus C \oplus
U$ we call $v^c\in C$ its projection along $S\oplus U$. If $\dim C = c$ we have,
after perhaps rearranging the basis, that $e_1^c,\dots,e^c_c$ is a basis of
$C$.

Using this basis to identify $C$ with $\mathbb{R}^c$ we can write the
projection to $C$ as a linear map $P:\mathbb{R}^d\to \mathbb{R}^c$, with $P(e_i)=e_i$
for $i=1,\dots,c$. With this notations we can write (in the linear case) the central
translations defined in \S~\ref{sec:trans} as $T_na=a+Pn$, and we can show that they satisfy the
following Diophantine condition:
\begin{Def}\label{diophantine condition}

  Let a linear map $P:\mathbb{Z}^d\to \mathbb{R}^c$ induce a
  $\mathbb{Z}^d$-action by translations ($n\cdot a = a + Pn$). We say that such
  $P$ is \emph{Diophantine} if, for some constants $K>0$ and $\tau>0$,
  for any integer vector $q\in \mathbb{Z}^c\setminus\{0\}$ there is some
  $i=1,\dots,d$ so that for any $p\in \mathbb{Z}$
\begin{equation*}
  \lvert Pe_i\cdot q - p \rvert > \frac{K}{\lvert q \rvert ^{\tau}}.
\end{equation*}
\end{Def}

After showing this we will show, using Moser in \S~\ref{fede linearization},
that the general central translations (of a map $f$ that is a jointly
integrable perturbation of a linear map $F$ instead of a linear map
itself) is conjugated to the Diophantine translations of the nearby
linear map.

\begin{thm}\label{katzthendio}

  For a Katznelson irreducible automorphism $F\in SL(d, \mathbb{Z})$, writing its
  central translations as $T_na=a+Pn$ as describe before we have that
  $P$ is Diophantine with $\tau=c$, the dimension of the central
  distribution for the partial hyperbolic splitting. That is
  \begin{equation*} \lvert
    Pe_i\cdot q - p \rvert > \frac{K}{\lvert q \rvert ^{c}}
\end{equation*} with the existential quantifiers of definition
\ref{diophantine condition}.

\end{thm}

During the proof we will actually use Katznelson's lemma for $F^t$.
\begin{lem}\label{diophantine-ineq}

  If $F\in SL(d,\mathbb{Z}^d)$ is Katznelson irreducible then for any integer
  vector $\bm{p}\in \mathbb{Z}^d$ \begin{equation*}
 d(\bm{p},(S\oplus U)^{\perp}) > \frac{\hat{K}}{\lvert \bm{p} \rvert^{c}}.
\end{equation*}

\end{lem}

\begin{proof}

  Of course we just need to show that we can apply Katznelson's lemma.
  Namely we need to show $(S\oplus U)^{\perp} \cap \mathbb{Z}^d
  = \{0\}$ and that $F^t|_{(S\oplus U)^{\perp}}$ and $F^t|_{C^{\perp}}$ have no common
  eigenvalue.

  This is clear from the characterization by polynomials in
  lemma~\ref{poly} and the fact that $F^t|_{(S\oplus U)^{\perp}}$ and $F|_C$,
  and  $F^t|_{C^{\perp}}$ and $F|_{S\oplus U}$ have the
  same eigenvalues.
\end{proof}

\begin{proof}[Proof of Theorem~\ref{katzthendio}]
 The Diophantine condition is equivalent to the existence of an $i$ so that
\begin{equation*}
  \lvert e_i\cdot P^tq - p \rvert > \frac{K}{\lvert q \rvert^{\tau}}
\end{equation*} which in turn is equivalent to
\begin{equation*}
  \lVert P^t q - \bm{p} \rVert > \frac{K}{\lvert q \rvert^{\tau}}
\end{equation*} for any $\bm{p} \in \mathbb{Z}^d$ (indeed picking each coordinate
of $\bm{p}$ to be the closest integer to the respective coordinate of
$P^tq$ is easy to conclude the former inequality).

But (using lemma \ref{diophantine-ineq} for the last inequality)
\begin{equation*}
   \lVert P^t q - \bm{p} \rVert \ge d(\bm{p}, \operatorname{Im}P^t)=d(\bm{p},
   (\operatorname{Ker}P)^\perp)= d(\bm{p},(S\oplus U)^{\perp}) )> \frac{\hat{K}}{\lvert \bm{p} \rvert^{c}}.
\end{equation*}
In summary
\begin{equation*}
  \lVert P^t q - \bm{p} \rVert>  \frac{\hat{K}}{\lvert \bm{p} \rvert^{c}}.
\end{equation*}

Now we need to change the $\bm{p}$ for the $q$ in the right hand expression. If $\lVert P^t q - \bm{p}
\rVert>0.5$ we can just adjust the $K$ to be smaller than $0.5$ (notice
$\lvert q \rvert \geq 1$), otherwise we have that  the norms of
$P^tq$ and $\bm{p}$ cannot be too far, so $\lvert \bm{p} \rvert < 2
\lvert P^tq \rvert < 2 \lVert P^t \rVert\lvert q \rvert$, finally concluding
\begin{equation*}
  \lVert P^t q - \bm{p} \rVert>  \frac{\hat{K}}{\left( 2 \lVert P^t \rVert\right)^c}\frac{1}{\lvert \bm{q} \rvert^{c}}.
\end{equation*}
\end{proof}

\subsection{Using Moser to Linearize the Central Translations}\label{fede linearization}

\begin{thm}\label{lin}

  For a chosen integer regularity $\hat{k}\in \mathbb{N}$, if $f \in C^{R}$ is $C^{6c+1}$-close to $F$
  and $F$ is $R$-bunched, with $R\ge3\hat{k}+3c
  + 2$ and $R \ge9c+4$, there exist $h: \mathcal{C} \to
  \mathbb{R}^c$, which is $C^{\hat{k}}$, so
  that 
\begin{equation*}
  h\circ T_n \circ h^{-1}(x)= x+ Pn 
\end{equation*} where $P: \mathbb{Z}^d \to \mathbb{R}^c$ is exactly as defined in
\S~\ref{linear-thm} for the nearby linear map $F$.

\end{thm}

This is essentially already done in
\cite{RodriguezHertz2005StableErgodicityCertain}, we sketch the
construction below and explain the regularities.

Recall than in \S~\ref{sec:trans} we projected $T_n$ to maps
$\tilde{T}_n:C \to C$ of a vector space. We also make the identification
$C=\mathbb{R}^c$ exactly like we did in \S~\ref{linear-thm}.

By means of bump functions we can construct
$\tilde{h}:\mathbb{R}^c\rightarrow \mathbb{R}^c$ so that
$$\tilde{h} \circ \tilde{T}_{\ell}\circ \tilde{h}^{-1}= R_{\ell}$$ for
$\ell\in \mathbb{Z}^c \times \{0\} \subset \mathbb{Z}^d$, where $R_{\alpha}=x\mapsto x+\alpha$, and for a general
$n\in \mathbb{Z}^d$
$$\tilde{h}\circ  \tilde{T}_n\circ  \tilde{h}^{-1}\approx_{C^{6c+1}} R_{Pn}$$
(the proximity to the respective translation for the linear map comes from
 \cite[Corollary 2.4]{RodriguezHertz2005StableErgodicityCertain}).

To improve this $\tilde{h}$ into a simultaneous linearization we
follow Moser \cite{Moser1990CommutingCircleMappings}. Moser shows how
to simultaneously linearize commuting perturbations of rotations of
the circle with a Diophantine condition analogous to ours by using some
KAM type of arguments.

Essentially all the difficulties of
generalizing this to general dimension are already taken care by
F.~Rodriguez-Hertz (the only minor detail is that in the proof of
\cite[Lemma B.3]{RodriguezHertz2005StableErgodicityCertain} it is used
the fact that $F|_C$ is an isometry, but actually the argument works
just as well with the hypothesis of partial hyperbolicity).

In here we will follow the constants in Moser to see how much they can
be optimized. Clearly in our case $\tau=c$. In lemma 3.1 from Moser there is a
convergence that is over $\mathbb{Z}^c$ in our case, making it so that we need
$\sigma>2\tau+c=3c$.

In page 119 we leave, for now, the variables \(\kappa\) and \(\ell\)
free, and set
$N_s=\varepsilon_s^{-\xi}$ with $\xi$ free. The conditions needed to make the rest
of the page work would be 
\begin{align*}
  \xi\sigma +1 > \kappa, \\
  -\xi +2 - \frac{2}{\ell} > \kappa, \\
  \ell\xi -1 >\kappa.
\end{align*}

For the next page we leave the variable $m$ free and we need the
argument to work for $k$ up to $\hat{k}$, so the condition for the argument
to work would be 
\begin{equation*}
  -\xi\sigma +1 - \frac{\hat{r}}{m}-\frac{\hat{r}}{m}\frac{\xi\sigma+\kappa-1}{\kappa-1}>0.
\end{equation*}

For the argument to work we need existence of derivatives up to order $\ell+\sigma$ and
$m+\sigma$ and to control derivatives up order \(\ell\). So we
need to optimize those numbers.

Taking $\kappa$ close enough to $1$ and $\xi$ close enough to $\frac{\kappa-1}{\sigma}$
the conditions are satisfied for 
\begin{align*}
  \ell > 2\sigma,\\
  m > 3\hat{k}.
\end{align*}

Since \(\sigma\), \(\ell\) and $m$ need to be integers we can solve the
problem with $\sigma=3c+1$, \(\ell=6c+3\) and $m=3\hat{k}+1$.

\subsection{Using Fourier to Solve Cohomological Equations for the
  Central Translations}\label{sec:fourier}

So far we show that, in a sense, the central translations are
Diophantine. We now show that Diophantine translations have no
restrictions to solve certain cohomological equations.

\begin{thm}\label{affine}
  Let the linear map $P:\mathbb{Z}^d\to \mathbb{R}^c$, inducing a $\mathbb{Z}^d$-action by
  translations, be Diophantine (see definition \ref{diophantine
    condition}). Further assume that $P\ell=\ell$ for any $\ell\in \mathbb{Z}^c\times \{0\} \subset
  \mathbb{Z}^d$.

  Then any $C^{r+2c+ \text{Hölder}}$ cocycle
  $\Phi:\mathbb{Z}^d\times \mathbb{R}^c\to \mathbb{R}$ for the induced action is
  $C^{r}$-cohomologus to a group morphism
  $\lambda:\mathbb{Z}^d \to \mathbb{R}$, that is there is
  $u \in C^r(\mathbb{R}^c,\mathbb{R})$ so that $$\Phi(n,a)=u(a+Pn)-u(a)+\lambda (n).$$

  Furthermore $\lambda(n)=\int_{[0,1]^c}\Phi(n,a)\,da$.
\end{thm}

We remark that the condition $P\ell=\ell$ can be replaced by $\Phi$ being
defined on the torus $\mathbb{T}^c$.

\begin{proof}
  For this proof we write $\mathbb{Z}^c=\mathbb{Z}^c\times \{0 \}\subset \mathbb{Z}^d$ so that $P\ell=\ell$ for
  $\ell\in \mathbb{Z}^c$. Also we will call ``$K$'' to different constants.
  
  We first project to the torus
  $\mathbb{T}^c=\mathbb{R}^c/\mathbb{Z}^c$ by constructing
  $v:\mathbb{R}^c \to \mathbb{R}$ so that
  $\tilde \Phi (n,a)= \Phi(n,a) + v(a+Pn)-v(a)$ satisfies that
  $\tilde \Phi (n,a+\ell)=\tilde{\Phi}(n,a)$ for $\ell\in \mathbb{Z}^c$.

  Soon we will define $v$ for $a\in[0,1)^c$, after that we will extend
  it for $a+\ell\in\mathbb{R}^c$, $\ell\in \mathbb{Z}^c$, by
  $$v(a+\ell)=v(a) - \Phi ( \ell,
  a)$$ This way we get for
  $\ell\in \mathbb{Z}^c$ \begin{align*}\tilde{\Phi}(n,a+\ell)&=\underbrace{\Phi(n,a+\ell)}_{\Phi(n+\ell,a)-\cancel{\Phi(\ell,a)}}+\underbrace{v(a+\ell+Pn)}_{v(a+Pn) -\Phi(\ell,a+Pn)}\underbrace{-v(a+\ell)}_{-v(a)+\cancel{\Phi(\ell,a)}}\\
                                       &=
                                         \underbrace{\Phi(n+\ell,a)-\Phi(\ell,a+Pn)}_{\Phi(n,a)}
                                         + v(a+Pn)-v(a)\\ &= \tilde{\Phi}(n,a).\end{align*}

                                       We only need to choose a smooth
                                       \(v:[0,1)^c\to \mathbb{R}\) so that its
                                       extension to $\mathbb{R}^c$ is smooth at
                                       the boundaries of $[0,1]^c$.
                                       For this consider
                                       $b:[0,1] \to [0,1]$ a bump
                                       function ($C^\infty$, constantly
                                       $0$ around $0$ and constantly
                                       $1$ around $1$) and just
                                       define
                                       $$v(a)=-\sum b(a_i)\Phi(e_i,a-e_i)$$ where
                                       $e_i\in \mathbb{Z}^c$ are the canonical
                                       elements and $a_i$ is the
                                       $i$-th coordinate of $a$. With
                                       this definition we can check
                                       that \(v\), when some $a_i$
                                       tends to $1$, coincides with the
                                       definition given by the
                                       extension, from there
                                       smoothness is clear.

  We can consider now that $P$ induces an action on the torus $\mathbb{T}^c$
  with cocycle $\tilde{\Phi}$ cohomologus to the original cocycle in $\mathbb{R}^c$. So now we have effectively ``projected'' our problem to the torus,
  where we can use Fourier:
  $$\tilde \Phi(n,a) = \sum_{\ell\in \mathbb{Z}^c} \hat \Phi_{n,\ell} e^{2\pi i \ell \cdot a}.$$
  The cocycle condition
  ($\tilde \Phi (n+m,a)=\tilde \Phi(n,a) + \tilde \Phi (m, a+ Pn)$) now
  becomes
  $$\hat \Phi_{n+m,\ell} = \hat \Phi_{n,\ell} + \hat \Phi_{m,\ell}e^{2\pi i \ell
    \cdot Pn}$$ which implies
  $\hat \Phi_{n,\ell} + \hat \Phi_{m,\ell}e^{2\pi i \ell \cdot Pn}= \hat \Phi_{m,\ell} + \hat
  \Phi_{n,\ell}e^{2\pi i \ell \cdot Pm}$ because we can commute the roles on $n$ and
  $m$. From there
  \begin{equation} \label{coefficents}\frac{\hat \Phi_{n,\ell}}{e^{2\pi i \ell \cdot
        Pn}-1}=\frac{\hat \Phi_{m,\ell}}{e^{2\pi i \ell \cdot Pm}-1}.
  \end{equation}

  Now we want to solve the original problem in this context: that is
  finding $w$ and $\lambda$ so that
  $$\tilde \Phi(n,a)=w(a+Pn)-w(a)+\lambda(n).$$ Of course
  $\lambda(n) = \int \tilde \Phi (n,a)\,da$ (the cocycle condition for
  $\tilde \Phi$ translates into the group morphism condition for
  $\lambda$). Because translations preserve volume we have $\int
  \tilde \Phi (n,a)\,da = \int_{[0,1]^c}  \Phi (n,a)\,da$.

  As for the rest of the Fourier coefficients, they should
  satisfy
  $$\hat w_\ell = \frac{\hat \Phi_{m,\ell}}{e^{2\pi i \ell \cdot Pm}-1}$$ which is well
  defined in light of (\ref{coefficents}). Furthermore, picking the
  correct $i$ for the Diophantine condition~\ref{diophantine
    condition}, so that $\lvert e^{2\pi i\ell \cdot Pm-1} \rvert \ge Kd(\ell\cdot Pm, \mathbb{Z})
  > K \lvert \ell \rvert^{-c}$, and using $\lvert \hat{\Phi}_{e_i,\ell} \rvert\leq K \lVert \hat{\Phi}_{e_i} \rVert \lvert \ell
  \rvert^{-r-2c-\varepsilon}$ we have
  $$\left| \hat{w}_{\ell} \right|\le K \left| \hat{\Phi}_{e_i,\ell} \right|\left|
    \ell \right|^{c} \leq K \lVert \hat{\Phi}_{e_i} \rVert \lvert \ell
  \rvert^{-r-2c-\varepsilon}\lvert \ell \rvert^c\le K \lvert \ell \rvert^{-r-c-\varepsilon}$$
  so the Fourier series of $\hat{w}$ converges to a $C^{r+\text{Hölder}}$ function.

 We now think of this $w$ as defined in
  $\mathbb{R}^c$ and set $u=v+w$.
\end{proof}

\subsection{Minimality of the Central Translations}

Here is a small lemma that we will need later:

\begin{lem}\label{dense}

  Any orbit for the central translations $T_n: \mathcal{C} \to \mathcal{C}$ (as
  $n$ varies along $\mathbb{Z}^d$) is dense.

\end{lem}

\begin{proof}

  In virtue of lemma \ref{lin} we only need to prove this in the
  linear case where $T_n:\mathbb{R}^c\to \mathbb{R}^c$ is of the form $T_na=a+Pn$, where
  $P$ satisfies the Diophantine
  condition of definition~\ref{diophantine condition} and $P\ell=\ell$ for $\ell\in \mathbb{Z}^c\times \{0\}\subset \mathbb{Z}^d$.

  Because of the $P\ell=\ell$ condition we can think of $T_n$ as acting on
  the torus $\mathbb{T}^c$ and we just need to show that the set $\{Pn\}_{n\in
    \mathbb{Z}^d} \subset \mathbb{T}^c$ is dense.

  Because of classification of closed Lie groups is easy to see that,
  for some sub-torus $V\subset \mathbb{T}^c$ and finite subgroup $Q\subset \mathbb{T}^c$,
\begin{equation*}
  \overline{\{Pn\}}= V + Q.
\end{equation*}

If $V \subsetneq \mathbb{T}^d$ then we can get an integer vector $p\in \mathbb{Z}^c$ perpendicular
to $V$ and (by multiplying for example by the least common multiple of
the denominators of all the coordinates of all elements of $Q$) so
that $p \cdot q \in \mathbb{Z}$ for every $q\in Q$.

So now for every $n\in \mathbb{Z}^d$ we have $Pn=v+q$ for some $v\in V$ and $q \in Q$, but
then $Pn \cdot p = q\cdot p \in \mathbb{Z}$, but this contradict the Diophantine
condition and so $V=\mathbb{T}^d$ and the orbit of $0$ (and hence any orbit) is dense.
\end{proof}

\section{Proof of Theorem~\ref{main}}

\subsection{Solving the Cohomological Equation Along a Central Leaf}\label{sec:solved-central}

\subsubsection{Constructing the Solution}
With our work so far we can solve cohomological equations for the
central translations $T_n$. This begs to question: how is this
relevant to the original dynamic $f$?

The idea now is to capture the information of the holonomy functional
 into a cocycle for the central
translations:
 $$\hol(n,x)=\hol_{x+n}^{T_nx}\varphi.$$ Indeed this is
a cocycle for $T_n$ (in what follows we use that the $su$-foliation and
$\varphi$ are $\mathbb{Z}^d$-invariant, and \eqref{additivity}):
\begin{align*} \hol(n+m,x)&= \hol_{x+n+m}^{T_{n+m}x} \varphi\\ & =
  \left(\hol_{x+n+m}^{T_mx+n} +\hol_{T_mx+n}^{T_{n+m}x}\right)\varphi\\ &
                                                                       = \left(\hol_{x+m}^{T_mx} +\hol_{(T_mx)+n}^{T_n(T_mx)}\right)\varphi \\
                                 &= \hol(m,x) +
                                   \hol(n,T_mx).
\end{align*}

Using theorem~\ref{lin} we have $h$ such that
$h(T_nx)=h(x)+Pn$, with $P$ Diophantine in the sense of definition~\ref{diophantine condition}. So
$$ \widehat{\hol} (n,a)=\hol \left(n, h^{-1}(a)\right)$$ is a cocycle for the action of
$P$ by translations. So we can apply theorem~\ref{affine} getting: $$ \widehat{\hol}(n,a)=\hat{u}(a+Pn)-\hat{u}(a)+\lambda(n)$$ which in turn gives us an
analogous result for the $T_n$-cocycle:
\begin{equation}\label{eq:solved-hol}
\hol(n,x)=u(T_nx)-u(x)+\lambda(n) 
\end{equation}
where $u=\hat{u}\circ h$.

As far as regularities go, this $u$ comes from theorem~\ref{affine},
so it would be $C^{k+\text{Hölder}}$ if
$\widehat{\hol}$ and $h$ are $C^{k+2c+\text{Hölder}}$. From lemma~\ref{holreg} we
need $F$ to be $(k+2c+\varepsilon)$-strongly bunched for
$\hol_{x+n}^{T_nx}\varphi=\hol_{x+n}^{\pi^c(x+n)}\varphi$ to be this regular (on top
of \(\varphi\) itself being this regular).

On the other hand $h$ comes from theorem~\ref{lin}, so we need $f$ to be
$C^{6c+1}$-close to $F$ and $F$ to be both $(3k+9c+3)$ and
$(9c+4)$-bunched.

\subsubsection{Properties of The Solution}

Several of the obvious computations needed in the rest of this notes would
have a ``troublesome'' $\lambda$, so we get it out of the way already by
proving that it is $0$.

\begin{lem}

  In equation \eqref{eq:solved-hol} we have $\lambda=0$.

\end{lem}

\begin{proof}

  We know that $\lambda(n)=\int
  \widehat{\hol}_{[0,1]^c}(n,a)\,da$  from lemma~\ref{affine}. So our strategy is to show that
  $\hol_{x+n}^{T_nx}$ grows sub-linearly on $n$ (uniformly so in $x$),
  implying that in order for $\lambda$ to be linear it has no other choice
  than to be $0$.

  We call $S_nx=\mathcal{W}^s(x+n)\cap \mathcal{W}^u(T_nx)$, that is the
  point so that
  \begin{align*}\hol_{x+n}^{T_nx}\varphi= &\left( \sum\limits_0^{\infty}\varphi(f^i(x+n))-\varphi(f^iS_nx)
  \right)\\&- \left( \sum\limits_1^{\infty}\varphi(f^{-i}(S_nx))-\varphi(f^{-i}T_nx) \right).\end{align*}

  We will prove that
  $ \sum_0^{\infty}\varphi(f^i(x+n))-\varphi(f^iS_nx)$ is sub-linear in
  $n$ (the other half is analogous). It is clear, since all leaves are
  close to the linear leaves, that the distance between $x+n$ and
  $S_nx$ is linearly bounded in $n$, and since the stable manifold
  contracts exponentially there is $i(n)$ (logarithmically in $n$) so
  that $f^{i(n)+j}(x+n)$ and $f^{i(n)+j}S_nx$ are at distance at most $\mu^j$
  for $\mu<1$ the contraction rate of the partial hyperbolicity. 
  So $$ \sum\limits_{i(n)}^{\infty}\varphi(f^i(x+n))-\varphi(f^iS_nx)<K$$ and
  $$\hol_{x+n}^{T_nx}\varphi<2i(n)\max\varphi +K.$$
\end{proof}

Now we can solve the cohomological equation along the central leaf:
\begin{lem}\label{centralsol}

  The $u$ of equation \eqref{eq:solved-hol}
  satisfies $$\varphi(x)=u(fx)-u(x)+ \mathrm{const.}$$ for any $x\in \mathcal{C}$.

\end{lem}

\begin{proof}We will call $u_f=u\circ f - u$ (so the objective is to prove
  that $u_f-\varphi$ is constant).

Before the computations notice that $f(x+n)=fx +Fn$, where $F$ is the
nearby linear map, and
$fT_nx=T_{Fn}fx$ (since both lie in $\mathcal{W}^{su}(fx+Fn)\cap \mathcal{C}$). We will
further use equation \eqref{fundamental} in one of the steps.
\begin{align*}
  u_f(T_nx)-u_f(x) & = u(fT_nx)-u(fx)-(u(T_nx)-u(x))  \\
                   &= u(T_{Fn}fx)-u(fx) - \hol_{x+n}^{T_nx}\varphi   \\
                   &= \hol_{fx+Fn}^{T_{Fn}fx}\varphi  - \hol_{x+n}^{T_nx}\varphi 
  \\&= \left(\hol_{f(x+n)}^{fT_nx}-\hol_{x+n}^{T_nx}\right)\varphi  \\
                   &= \varphi(T_n x) - \varphi(x+n)   \\
                   &= \varphi(T_n x) - \varphi(x). 
\end{align*}  

So we have concluded $$u_{f}(x)-\varphi(x)=u_f(T_nx)-\varphi(T_nx),$$ so from
continuity and minimality of the action of $T_n$ (lemma \ref{dense}) we conclude that
indeed $u_f-\varphi$ is constant. 
\end{proof}

\subsection{Extending the Solution}\label{extend}

So after \S~\ref{sec:solved-central} we have $u: \mathcal{C} \to \mathbb{R}$ that satisfies the
cohomological equation \eqref{eq:cohomological}. We just need to
extend it, first to $\mathbb{R}^d$ because of our abuse of notation and then
show that it projects to $\mathbb{T}^d$ (and that it is indeed a solution).
Such extension will be
\begin{equation*}
  u(x)=u(\pi^cx) - \hol_{x}^{\pi^{c}x}\varphi
\end{equation*} where the right hand side $u$ is well defined since
$\pi^cx\in \mathcal{C}$.

\begin{lem}

  This $u$ satisfies the cohomological equation \eqref{eq:cohomological}.

\end{lem}

\begin{proof} First notice that $\pi^cfx=f\pi^cx$ since the foliations are
  $f$-invariant. In what follows we will use lemma \ref{centralsol}
  and the ``fundamental'' theorem \eqref{fundamental}.
  \begin{align*}
  u(fx)-u(x)&=u(f\pi^cx)-u(\pi^cx)-\left(\hol_{fx}^{f\pi^cx}-\hol_{x}^{\pi^cx}\right)\varphi\\
            &=\varphi(\pi^cx)-K-(\varphi(\pi^cx)-\varphi(x))\\ &=\varphi(x)-K.
\end{align*}
\end{proof}

\begin{lem}

  This $u$ is well defined on the torus $\mathbb{T}^d$, that is $u(x+n)=u(x)$
  for every $n\in \mathbb{Z}^d$.

\end{lem}

\begin{proof}

  First we notice that $\pi^c(x+n)=T_n\pi^cx$ since both lie in the intersection
  $\mathcal{W}^{su}(x+n)\cap \mathcal{C}$. We will also use the $\mathbb{Z}^d$-invariance of \(\varphi\),
  the definition of $u$ along $\mathcal{C}$ given in \S~\ref{sec:solved-central}
  (that is $u(T_nx)-u(x)=\hol_{x+n}^{T_nx}\varphi$), and \eqref{additivity}.
\begin{align*}
  u(x+n)- u(x) &= u(\pi^c(x+n)) - u(\pi^cx) -\left( \hol_{x+n}^{\pi^c(x+n)} -
    \hol_{x}^{\pi^cx} \right)\varphi\\
  &= u(T_n\pi^cx) - u(\pi^cx) -\left( \hol_{x+n}^{T_n\pi^cx} -
    \hol_{x+n}^{\pi^cx+n} \right)\varphi\\
  &= \hol_{\pi^cx +n}^{T_n\pi^cx}\varphi - \hol_{\pi^cx +n}^{T_n\pi^cx}\varphi\\&= 0
\end{align*} \end{proof}

Let us clarify the regularity of the $u$. If we are in conditions for
the $u:\mathcal{C} \to \mathbb{R}$ to be $C^{r+\text{Hölder}}$ described in \S~\ref{sec:solved-central},
and using the second part of lemma~\ref{holreg} then
the extended $u$ is uniformly $C^{r+\text{Hölder}}$ along $\mathcal{W}^c$.

The only dependence on the direction $su$ of our $u$ is the top $x$ in
$\hol_{x_0}^x\varphi$, so using the first part of lemma~\ref{holreg} we have
the desired regularity along this direction.

So using Journé's lemma \cite{Journe1988RegularityLemmaFunctions} we
have that $u$ is $C^r$ if we have enough regularities and bunching.

\appendix
\section{Appendix}
\subsection{Regularities for Theorem~\ref{main}}\label{optimal}

When trying to apply our theorem one would need to know how much
bunching is needed, so below we clarify exactly the bunching and
regularities needed.

\begin{manualtheorem}{A'}\label{regularities} For this theorem we call
  $c_p$ the dimension of the central distribution for the partial
  hyperbolic splitting and $c_1$ the dimension of the generalized
  eigenspace of eigenvalues of modulus $1$ for the nearby linear map $F$.

  Fix a desired integer regularity $k\in \mathbb{N}$ and define the minimum
  regularity bounds as 
\begin{align*}
  R_{b} &\ge \begin{cases} \max\{3k+9c_1+3,9c_p+4\}  \quad&\text{if }c_1\neq0,\\ \max\{k+\varepsilon,
            9c_p+4\} \quad&\text{if }c_1=0 \end{cases}, \\
  R_{sb} &> \max\{k+2c_1,2c_p\}.\end{align*}

  In theorem~\ref{main} if we want $u\in C^{k+\text{Hölder}}$ the exact regularities needed are as follow:

 \begin{description} 
 \item[General Case] \phantom{T} 

   \begin{itemize}
  \item $F$ needs to be $R_b$-bunched and $f$ needs to be
  $C^{R_b}$,

  \item $F$ needs to be strongly $R_{sb}$-bunched and $\varphi$ needs to be
  $C^{R_{sb}}$, and
\item $f$ needs to be $C^{6c_p+1}$-close to $F$.
\end{itemize}

\item[Linear Case] In the case where $f=F$ is linear
\begin{itemize}
  \item $F$ needs to be strongly $R_{sb}$-bunched and $\varphi$ needs to be
  $C^{R_{sb}}$.
  \end{itemize}
Or
 \begin{itemize}
  \item $F$ needs to be strongly $(2c_p+\varepsilon)$-bunched, and
  \item $\varphi$ needs to be
    $C^{k+d+\text{Hölder}}$.
  \end{itemize}

\item[Low center dimension] If $c_p=1$ and $c_1=0$

   \begin{itemize}
  \item $F$ needs to be $(4+\varepsilon)$-bunched and $f$ needs to be
  $C^{4+ \text{Hölder}}$-close to $F$, and

  \item $F$ needs to be strongly $R_{sb}$-bunched and $\varphi$ needs to be
  $C^{R_{sb}}$.
\end{itemize}

 If $c_{p}=c_1=2$ and $F$ is irreducible with $d>4$ (or the
  irreducible factor containing the complex eigenvalues is of degree
  more than $4$)

     \begin{itemize}
  \item $f$ needs to be
  $C^{4k+8+\text{Hölder}}$-close to $F$, and

  \item $\varphi$ needs to be
  $C^{k+4+\text{Hölder}}$.
\end{itemize}
  
 \end{description}
\end{manualtheorem}

\begin{proof}

  Notice that we are always able to bootstrap the regularity of the
  solution by applying the theorem with $k=0$ first. After
  corollary~\ref{conjugacy} and remark~\ref{suregular} we have jointly
  integrability for any of the splittings coming from any partial
  hyperbolic splitting of $F$. So, once we know that
  \(\varphi\) is a coboundary, we can assume that the center direction is
  (the perturbation of) the generalized eigenspace of eigenvalues of
  modulus $1$ (or, if the case allows it, apply Livsic or Veech \cite{Veech1986PeriodicPointsInvariant}).

  In the general case is only a matter of following the regularities
  along the proof.

  For the linear case on one hand we don't need to use Moser to
  linearize the central translations. Also in some cases it might be
  better to bootsrap using Veech.

  In the low dimensional cases we can get the Diophantine property
  with just one translation (see \cite[Lemma
  A.11]{RodriguezHertz2005StableErgodicityCertain} for the case $c=2$), so we can bypass
  Moser and use a more traditional KAM theory instead (see
  \cite[\S~A.2]{Herman1979ConjugaisonDifferentiableDiffeomorphismes}
  and \cite[\S~6.1]{RodriguezHertz2005StableErgodicityCertain}).
\end{proof}

We can go even further saying that $\varphi$ only needs to be $C^{R_{sb}}$
along the central direction while in the $su$-direction it only needs
to be $C^k$. Furthermore if we ask $u$ to be $C^k$ along the
central direction and $C^{\ell}$ along the $su$-direction (with $\ell\leq k$)
we need \(\varphi\) to be $C^{R_{sb}}$ along the central direction and
$C^{\ell}$ along the $su$-direction. In here we asking for the
regularities to be uniform.

We remark that, after \cite[Proposition
1.5]{Veech1986PeriodicPointsInvariant}, at least some regularity is
expected to be lost. We do not claim that our specific regularity
loss and bunching are optimal, although probably more refined
techniques would be needed to make significant improvements in these numbers.

\subsection{A Couple of Toy Examples}\label{toy}

For this section we use the holonomy functionals defined in
\S~\ref{holonomy}.

The simplest example where we can use the idea of solving the equation
along $su$-leaves and along central leaves independently and then join
the solutions would be $(x,y)\to(x+\alpha, Ay)$ where \(\alpha\) is Diophantine
and $A$ is Anosov. For the sake of generality we prove something stronger.

\begin{thm}

  Let $f:M\to M$ be partially hyperbolic so that
  \eqref{eq:cohomological} admits a solution for any \(\varphi\) with
  trivial periodic cycle functionals, and let $\tilde{f}:M\times N \to M\times N$ and
  hyperbolic extension 
\begin{equation*}
  \tilde{f}(x,y)=(fx,g_xy)
\end{equation*} so that there is a probability measure $\mu$ preserved
by $g_x$ for every $x$ (and so that the expansion and contraction
along $N$ dominate the center direction of $f$).

Then $\tilde{f}$ is cocycle stable.

\end{thm}

\begin{thm}

  Let $f: M \to M$ be cocycle stable and consider the suspension space
  $\tilde{M}= (\mathbb{R} \times M)/f$ where $(x,y)\sim (x-1,fy)$. For $\alpha$ Diophantine
  the map $\tilde{f}: \tilde{M} \to \tilde{M}$ 
\begin{equation*}
  \tilde{f}(x,y)=(x+\alpha, y)
\end{equation*} is cocycle stable

\end{thm}

The proof for both theorems is the same (that is the reason for the
awkward notation in the second theorem).
\begin{proof}
 For the second theorem consider $\mu$ an invariant measure for $f$.
  
  If \(\varphi\) has trivial periodic cycle functionals then it easy to see
  that so does 
\begin{equation*}
  \varphi_{\mu}(x) = \int_N \varphi(x,z)\,d\mu(z)
\end{equation*} and so we can solve 
\begin{equation*}
  \varphi_{\mu}(x)=u_{\mu}(fx) - u_{\mu}(x) +K.
\end{equation*}

Now define 
\begin{equation*}
  u(x,y)=u_{\mu}(x) + \int_N\hol_{(x,z)}^{(x,y)}\varphi \,d\mu(z).
\end{equation*}
Now, using (\ref{fundamental}), we show
\begin{align*}
  u(\tilde{f}(x,y))-u(x,y)&=u_{\mu}(fx)-u_{\mu}(x)+\int_{N}\left(\hol_{\tilde{f}(x,z)}^{\tilde{f}(x,y)}-\hol_{(x,z)}^{(x,y)}\right)\varphi \,d\mu(z)
  \\
                  &=\int_N \varphi(x,z)\,d\mu(z)-K+\int_{N}(\varphi(x,y)-\varphi(x,z))\,d\mu(z)\\ &=\varphi(x,y)-K.
\end{align*}
\end{proof}

And interesting example of the first theorem would be and automorphism
$F$ of a nilmanifold $N$ so that the fibers of the map $\pi:N \to \mathbb{T}^a $ are
hyperbolic and so that the projection of $F$ to $\mathbb{T}^a$ is Katznelson
irreducible. These properties are easy to check: if 
\begin{equation*}
  DF = \begin{pmatrix} A_1 & & & & \\
                           & A_2 & & \mbox{\huge $\ast$}  &\\
                           & & \ddots & &\\
 &  \mbox{\huge $0$}      & & A_{k-1}     \\
         
                      & & & &A_k
 \end{pmatrix}
\end{equation*}
is so that $A_i$ is hyperbolic for $i=1,\dots, k-1$ and $A_k$ is
Katznelson irreducible, with respect to the partial hyperbolic
splitting which should have its center direction on the coordinates
corresponding to $A_k$, then $F$ is cocycle stable.

This proof actually work for any map that has a cocycle stable factor
and whose fibers are hyperbolic (dominating the center of the factor)
with ``good enough'' transversal measures.

``Good enough'' might be a smooth volume for instance. We can alway
use axiom of choice to create the transversal measures but in this
case the solution $u$ wouldn't even be measurable.

\printbibliography

@article{Herman1979ConjugaisonDifferentiableDiffeomorphismes,
  title = {Sur La Conjugaison Diff\'erentiable Des Diff\'eomorphismes Du Cercle \`a Des Rotations},
  author = {Herman, Michael R.},
  year = {1979},
  journal = {Publications Math\'ematiques de l'IH\'ES},
  volume = {49},
  pages = {5--233},
  issn = {1618-1913},
  urldate = {2023-04-25},
  langid = {english}
}

@article{Hirsch.Pugh.ea1970InvariantManifolds,
  title = {Invariant Manifolds},
  author = {Hirsch, M. W. and Pugh, C. C. and Shub, M.},
  year = {1970},
  month = sep,
  journal = {Bulletin of the American Mathematical Society},
  volume = {76},
  number = {5},
  pages = {1015--1019},
  publisher = {{American Mathematical Society}},
  issn = {0002-9904, 1936-881X},
  urldate = {2023-04-29}
}

@article{Journe1988RegularityLemmaFunctions,
  title = {A {{Regularity Lemma}} for {{Functions}} of {{Several Variables}}},
  author = {Journ{\'e}, Jean-Lin},
  year = {1988},
  month = aug,
  journal = {Revista Matem\'atica Iberoamericana},
  volume = {4},
  number = {2},
  pages = {187--193},
  issn = {0213-2230},
  doi = {10.4171/rmi/69},
  urldate = {2023-04-25},
  langid = {english}
}

@article{Katok.Kononenko1996CocycleStabilityPartially,
  title = {Cocycle Stability for Partially Hyperbolic Systems},
  author = {Katok, A. and Kononenko, A.},
  year = {1996},
  journal = {Math. Res. Letters},
  volume = {3},
  pages = {191--210}
}

@article{Katznelson1971ErgodicAutomorphismsTn,
  title = {Ergodic Automorphisms of {{Tn}} Are {{Bernoulli}} Shifts},
  author = {Katznelson, Yitzhak},
  year = {1971},
  month = jun,
  journal = {Israel Journal of Mathematics},
  volume = {10},
  number = {2},
  pages = {186--195},
  issn = {1565-8511},
  doi = {10.1007/BF02771569},
  urldate = {2023-03-17},
  langid = {english}
}

@article{Moser1990CommutingCircleMappings,
  title = {On Commuting Circle Mappings and Simultaneous {{Diophantine}} Approximations},
  author = {Moser, J{\"u}rgen},
  year = {1990},
  month = sep,
  journal = {Mathematische Zeitschrift},
  volume = {205},
  number = {1},
  pages = {105--121},
  issn = {0025-5874, 1432-1823},
  doi = {10.1007/BF02571227},
  urldate = {2022-07-04},
  langid = {english}
}

@article{Pugh.Shub.ea1997HolderFoliations,
  title = {H\"older Foliations},
  author = {Pugh, Charles and Shub, Michael and Wilkinson, Amie},
  year = {1997},
  month = jan,
  journal = {Duke Mathematical Journal},
  volume = {86},
  number = {3},
  pages = {517--546},
  publisher = {{Duke University Press}},
  issn = {0012-7094, 1547-7398},
  doi = {10.1215/S0012-7094-97-08616-6},
  urldate = {2023-04-25}
}

@article{RodriguezHertz2005StableErgodicityCertain,
  title = {Stable Ergodicity of Certain Linear Automorphisms of the Torus},
  author = {Rodriguez Hertz, Federico},
  year = {2005},
  month = jul,
  journal = {Annals of Mathematics},
  volume = {162},
  number = {1},
  pages = {65--107},
  issn = {0003-486X},
  doi = {10.4007/annals.2005.162.65},
  urldate = {2022-07-03},
  langid = {english}
}

@article{Veech1986PeriodicPointsInvariant,
  title = {Periodic Points and Invariant Pseudomeasures for Toral Endomorphisms},
  author = {Veech, William A.},
  year = {1986},
  month = sep,
  journal = {Ergodic Theory and Dynamical Systems},
  volume = {6},
  number = {3},
  pages = {449--473},
  issn = {0143-3857, 1469-4417},
  doi = {10.1017/S0143385700003606},
  urldate = {2022-07-03},
  langid = {english}
}

@article{Wilkinson2013CohomologicalEquationPartially,
  title = {The Cohomological Equation for Partially Hyperbolic Diffeomorphisms},
  author = {Wilkinson, Amie},
  year = {2013},
  journal = {Ast\'erisque},
  number = {358},
  pages = {75--165},
  issn = {0303-1179},
  urldate = {2022-07-03},
  isbn = {9782856297780},
  mrnumber = {3203217}
}

\end{document}